\documentclass[a4paper, intlimits, reqno]{amsart}

\usepackage[english]{babel}
\usepackage[T1]{fontenc}
\usepackage[utf8]{inputenc}
\usepackage{graphicx}
\usepackage{caption}
\usepackage{amsmath}
\usepackage{amssymb}
\usepackage{MnSymbol}
\usepackage{amsthm}
\usepackage{mathrsfs} 
\usepackage{mathtools}
\usepackage{color}
\usepackage{pgfplots}
\usepackage{epstopdf}
\usepackage{listings}
\usepackage{enumitem}
\usepackage{multicol}
\usepackage{url}
\usepackage{dsfont}
\usepackage[numbers]{natbib}
\usepackage{prettyref}
\pgfplotsset{compat=1.14}

\newrefformat{def}{Definition \ref{#1}}
\newrefformat{rem}{Remark \ref{#1}}
\newrefformat{sect}{Section \ref{#1}}
\newrefformat{prop}{Proposition \ref{#1}}
\newrefformat{thm}{Theorem \ref{#1}}
\newrefformat{ex}{Example \ref{#1}}
\newrefformat{fig}{Figure \ref{#1}}
\newrefformat{app}{Appendix \ref{#1}}

\swapnumbers
\newtheoremstyle{dotless}{}{}{\itshape}{}{\bfseries}{}{}{}
\theoremstyle{dotless}
\theoremstyle{plain}
\newtheorem{thm}{Theorem}[section]

\newtheorem{prop}[thm]{Proposition}

\theoremstyle{definition}

\newtheorem{rem}[thm]{Remark}

\newcommand{\N} {\mathbb{N}}

\newcommand{\R} {\mathbb{R}}
\newcommand{\C} {\mathbb{C}}
\newcommand{\K} {\mathbb{K}}

\definecolor{mygreen}{rgb}{0,0.6,0}
\definecolor{mygray}{rgb}{0.5,0.5,0.5}
\definecolor{mymauve}{rgb}{0.58,0,0.82}

\begin{document}

\title[Matrix equation]{On the solvability of the matrix equation $(1+ae^{-\frac{\|X\|}{b}})X=Y$}
\author[K.~Kruse]{Karsten Kruse}
\address{TU Hamburg \\ Institut f\"ur Mathematik \\
Am Schwarzenberg-Campus~3 \\
Geb\"aude E \\
21073 Hamburg \\
Germany}
\email{karsten.kruse@tuhh.de}



\date{\today}
\begin{abstract}
The treated matrix equation $(1+ae^{-\frac{\|X\|}{b}})X=Y$ 
in this short note has its origin in a modelling approach to describe the nonlinear 
time-dependent mechanical behaviour of rubber. 
We classify the solvability of $(1+ae^{-\frac{\|X\|}{b}})X=Y$ in general normed spaces $(E,\|\cdot\|)$
w.r.t.\ the parameters $a,b\in\R$, $b\neq 0$, and give an algorithm to numerically compute its 
solutions in $E=\R^{m\times n}$, $m,n\in\N$, $m,n\geq 2$, equipped with the Frobenius norm. 
\end{abstract}

\maketitle

\section{Introduction}
In \cite{plagge2020} the common approach to extend hyperelastic models by a well-known Prony series is modified. 
In general, the classic approach using a Prony series for extension results in the need to identify a large number of parameters. 
The identification is usually an ill-posed problem. 
Therefore in \cite{plagge2020}, the authors restrict themselves to a single modified Prony element with a load-dependent relaxation time 
leading to an approach with only two parameters. 
Using an implicit Euler-approach (see \cite[Eq.\ (28), p.\ 8]{plagge2020}), 
solving the underlying matrix differential equation yields to 
\begin{equation}\label{eq:stress_discr}
\frac{\sigma_{\operatorname{v},k+1}-\sigma_{\operatorname{v},k}^{R}}{\Delta t}
=-\sigma_{\operatorname{v},k+1}
  \frac{\exp\bigl(\frac{\|\sigma_{\operatorname{v},k+1}\|}{\sigma_{c}}\bigr)}{\tau_{\operatorname{p}}}
 +\frac{\Delta \sigma_{\operatorname{el}}^{R}}{\Delta t},\quad k\in\N_{0},
\end{equation}
where $\sigma_{\operatorname{v}}\in\R^{3\times 3}$ is the deviatoric stress of the modified Prony element, 
$\sigma_{\operatorname{v},k}^{R}$ the rotated viscolelastic Cauchy stress, 
$\Delta\sigma_{\operatorname{el}}^{R}\in\R^{3\times 3}$ the incremental elastic driving stress, $\|\cdot\|$ the Frobenius norm, 
$\sigma_{c}>0$ the critical stress, $\tau_{\operatorname{p}}>0$ the relaxation timescale in the effective relaxation time 
and $\Delta t>0$ a time step.
Equation \eqref{eq:stress_discr} can be rewritten as 
\begin{equation}\label{eq:stress_discr_equiv}
\bigl(1+\tfrac{\Delta t}{\tau_{\operatorname{p}}}
\exp\bigl(\tfrac{\|\sigma_{\operatorname{v},k+1}\|}{\sigma_{c}}\bigr)\bigr)\sigma_{\operatorname{v},k+1}
=\Delta \sigma_{\operatorname{el}}^{R}+\sigma_{\operatorname{v},k}^{R},
\end{equation}
which has the general form 
\[
(1+a\exp(-\tfrac{\|X\|}{b}))X=Y
\]
with $X:=\sigma_{\operatorname{v},k+1}$, $Y:=\Delta \sigma_{\operatorname{el}}^{R}+\sigma_{\operatorname{v},k}^{R}$, 
$a:=\tfrac{\Delta t}{\tau_{\operatorname{p}}}$ and $b:=-\sigma_{c}$.

\section{Classification of the solvability in general normed spaces}
Let $(E,\|\cdot\|)$ be a normed space over the field $\K=\R$ or $\C$, $a,b\in\R$, $b\neq 0$, and $Y\in E$. 
We are searching for a solution $X\in E$ of the vector equation 
\begin{equation}\label{eq:matrix_eq}
(1+ae^{-\frac{\|X\|}{b}})X=Y.
\end{equation}
If we take norms on both sides of \eqref{eq:matrix_eq}, then we obtain the scalar equation
\begin{equation}\label{eq:norm_eq}
|1+ae^{-\frac{\|X\|}{b}}|\|X\|=\|Y\|.
\end{equation}

The solutions of \eqref{eq:matrix_eq} and \eqref{eq:norm_eq} are related in the following manner.

\begin{prop}\label{prop:relation_scalar_case}
Let $(E,\|\cdot\|)$ be a normed space, $Y\in E$ and $y:=\|Y\|$. 
\begin{enumerate}
\item[(a)] If $x\in[0,\infty)$ is a solution of $|1+ae^{-x/b}|x=y$ and $1+ae^{-x/b}\neq 0$,
then 
\begin{equation}\label{eq:solution_matrix_eq}
X:=(1+ae^{-x/b})^{-1}Y\in E
\end{equation}
fulfils $\|X\|=x$ and $X$ is a solution of $(1+ae^{-\frac{\|X\|}{b}})X=Y$.
\item[(b)] If $X\in E$ is a solution of $(1+ae^{-\frac{\|X\|}{b}})X=Y$, then $x:=\|X\|\in[0,\infty)$ is a solution of $|1+ae^{-x/b}|x=y$.
\end{enumerate}
\end{prop}
\begin{proof}
Statement (b) is obvious, so we only need to prove (a), which follows from
\begin{align*}
\|X\|&=|(1+ae^{-x/b})^{-1}|\|Y\|=|(1+ae^{-x/b})^{-1}|y=|(1+ae^{-x/b})^{-1}||1+ae^{-x/b}|x\\
     &=x.
\end{align*}
\end{proof}

Hence we can use the solutions of the scalar equation \eqref{eq:norm_eq} to obtain the solutions 
of the vector equation \eqref{eq:matrix_eq}. \emph{Lambert's W function} will turn out to be 
a useful tool to solve the scalar equation \eqref{eq:norm_eq} with respect to the parameters involved.
In the following we denote by $W_{0}\colon[-e^{-1},\infty)\to[-1,\infty)$ the upper branch 
and by $W_{-1}\colon[-e^{-1},0)\to(-\infty,-1]$ the lower branch of Lambert's W function. These 
branches are bijective functions and the union of the sets $\{W_{0}(z)\;|\;z\in[-e^{-1},\infty)\}$ 
$\{W_{-1}(z)\;|\;z\in[-e^{-1},0)\}$ is the set of all solutions $x\in\R$ of the equation $xe^{x}=z$ 
(see e.g.\ \cite{Corless1996}). 

\begin{prop}\label{prop:scalar_case}
Let $a,b\in\R$, $b\neq 0$, and $y\in[0,\infty)$. The equation 
\begin{equation}\label{eq:scalar_eq}
 |1+ae^{-x/b}|x=y
\end{equation}
has in $[0,\infty)$
\begin{enumerate} 
 \item[(a)] a unique solution if $a\geq 0$ and $b<0$,
 \item[(b)] a unique solution if $a\leq -1$ and $b<0$,
 \item[(c)] two solutions $x=0$ or $x=b\ln(|a|)$ if $-1<a<0$, $b<0$ and $y=0$,
 \item[(d)] three solutions if $-1<a<0$, $b<0$ and $0<y<(1+ae^{W_{0}(-e/a)-1})b(1-W_{0}(-e/a))$,
 \item[(e)] two solutions if $-1<a<0$, $b<0$ and $y=(1+ae^{W_{0}(-e/a)-1})b(1-W_{0}(-e/a))$,
 \item[(f)] a unique solution if $-1<a<0$, $b<0$ and $y>(1+ae^{W_{0}(-e/a)-1})b(1-W_{0}(-e/a))$,
 \item[(g)] a unique solution if $0\leq a\leq e^{2}$ and $b>0$,
 \item[(h)] a unique solution if $a>e^{2}$, $b>0$ and $y<(1+ae^{W_{-1}(-e/a)-1})b(1-W_{-1}(-e/a))$ 
 or $y>(1+ae^{W_{0}(-e/a)-1})b(1-W_{0}(-e/a))$,
 \item[(i)] two solutions if $a>e^{2}$, $b>0$ and $y=(1+ae^{W_{-1}(-e/a)-1})b(1-W_{-1}(-e/a))$ 
 or $y=(1+ae^{W_{0}(-e/a)-1})b(1-W_{0}(-e/a))$,
 \item[(j)] three solutions if $a>e^{2}$, $b>0$ and $(1+ae^{W_{-1}(-e/a)-1})b(1-W_{-1}(-e/a))<y<(1+ae^{W_{0}(-e/a)-1})b(1-W_{0}(-e/a))$,
 \item[(k)] a unique solution if $-1\leq a<0$ and $b>0$,
 \item[(l)] two solutions $x=0$ or $x=b\ln(|a|)$ if $a<-1$, $b>0$ and $y=0$,
 \item[(m)] three solutions if $a<-1$, $b>0$ and $0<y<-(1+ae^{W_{0}(-e/a)-1})b(1-W_{0}(-e/a))$,
 \item[(n)] two solutions if $a<-1$, $b>0$ and $y=-(1+ae^{W_{0}(-e/a)-1})b(1-W_{0}(-e/a))$,
 \item[(o)] a unique solution if $a<-1$, $b>0$ and $y>-(1+ae^{W_{0}(-e/a)-1})b(1-W_{0}(-e/a))$.
\end{enumerate}
\end{prop}
\begin{proof}
Let us define the continuous function
\[
 f\colon[0,\infty)\to[0,\infty),\;f(x):=|1+ae^{-x/b}|x,
\]
given by the left-hand side of the equation under consideration. First, we determine the conditions on 
$a,b\in\R$ and $x\geq 0$ such that $1+ae^{-x/b}\geq 0$. If $a\geq 0$, then $1+ae^{-x/b}>0$ and thus
$f(x)=(1+ae^{-x/b})x$ for all $b\in\R$ and $x\geq 0$. 
Now, let us turn to the case $a<0$. We have the equivalences 
\[
1+ae^{-x/b}\geq 0\;\Leftrightarrow\;e^{-x/b}\leq |a|^{-1}\;\Leftrightarrow\; x/b \geq  \ln(|a|)
\;\Leftrightarrow\; \begin{cases} x\geq  b\ln(|a|) &,\; b>0,\\x\leq b\ln(|a|) &,\; b<0.\end{cases}
\]
Furthermore, we note that 
\[
 b\ln(|a|)\geq 0\;\Leftrightarrow\; (b>0,\;|a|\geq 1)\;\text{or}\;(b<0,\;|a|\leq 1).
\]
Hence, if $b>0$ and $-1\leq a<0$, then $b\ln(|a|)\leq 0$ and
\[
f(x)=(1+ae^{-x/b})x,\;x\in[0,\infty),  
\]
if $b>0$ and $a<-1$, then $b\ln(|a|)>0$ and 
\[
f(x)=\begin{cases}-(1+ae^{-x/b})x &,\;x\in[0,b\ln(|a|)],\\(1+ae^{-x/b})x &,\;x\in(b\ln(|a|),\infty),\end{cases}
\]
if $b<0$ and $-1<a<0$, then $b\ln(|a|)>0$ and
\[
f(x)=\begin{cases}(1+ae^{-x/b})x &,\;x\in[0,b\ln(|a|)],\\-(1+ae^{-x/b})x &,\;x\in(b\ln(|a|),\infty),\end{cases}
\]
if $b<0$ and $a\leq -1$, then $b\ln(|a|)\leq 0$ and
\[
f(x)=-(1+ae^{-x/b})x,\; x\in[0,\infty). 
\]
The set of zeros $\mathcal{N}_{f}$ of $f$ in $[0,\infty)$ is $\mathcal{N}_{f}=\{0\}$, 
if either $a\geq 0$ or $b>0$ and $-1\leq a<0$ or $b<0$ and $a\leq -1$, and $\mathcal{N}_{f}=\{0,b\ln(|a|)\}$ 
if either $b>0$ and $a<-1$ or $b<0$ and $-1<a<0$. 
We note that $f$ is infinitely continuously differentiable on $[0,\infty)\setminus\mathcal{N}_{f}$.

$(a)$ In this case our claim follows from the intermediate value theorem since $f$ is strictly 
increasing on $[0,\infty)$, $f(0)=0$ and $\lim_{x\to\infty}f(x)=\infty$.

$(b)$ The first derivative fulfils for all $x\in(0,\infty)$
\[
 f'(x)=\underbrace{-a}_{\geq 1}\underbrace{(1-\frac{x}{b})}_{>1}\underbrace{e^{-x/b}}_{>1}-1>1-1=0,
\]
implying that $f$ is strictly increasing on $[0,\infty)$. Like in (a) this proves our claim by the intermediate value theorem 
because $f(0)=0$ and $\lim_{x\to\infty}f(x)=\infty$ as well. 

$(c)$-$(f)$ We start with the first derivative of $f$. We have 
\[
 f'(x)=\begin{cases}a(1-\frac{x}{b})e^{-x/b}+1&,\;x\in(0,b\ln(|a|)),
 \\-a(1-\frac{x}{b})e^{-x/b}-1&,\;x\in(b\ln(|a|),\infty).\end{cases}
\]
The function $f$ is strictly increasing on $(b\ln(|a|),\infty)$ since for all $x>b\ln(|a|)$ it holds that
\begin{align*}
f'(x)&=-a(1-\frac{x}{b})e^{-x/b}-1>|a|(1-\frac{b\ln(|a|)}{b})e^{-b\ln(|a|)/b}-1
=|a|(1-\ln(|a|))|a|^{-1}-1\\&=-\ln(|a|)>0.
\end{align*}
Next, we compute the local extremum of $f$ on $(0,b\ln(|a|))$. 
For the second derivative on $(0,b\ln(|a|))$ we remark that
\[
 f''(x)=\underbrace{-\frac{a}{b}}_{<0}\underbrace{(2-\frac{x}{b})e^{-x/b}}_{>0}<0,\;x\in(0,b\ln(|a|)),
\]
so our local extremum will be a maximum.
Setting $z:=1-\frac{x}{b}>0$ for $x\in(0,b\ln(|a|))$, we have the equivalences
\[
 0=f'(x)\quad\Leftrightarrow\quad 0=aze^{-1}e^{z}+1\quad\Leftrightarrow\quad -\frac{e}{a}=ze^{z}.
\]
Since $z>0$ and $-\frac{e}{a}>0$, the unique solution of the last equation is given by $z=W_{0}(-e/a)$. 
This yields by resubstitution the local maximum $x_{0}:=b(1-W_{0}(-e/a))$, which is in $(0,b\ln(|a|))$ by Rolle's theorem.
Furthermore, $f$ has two zeros in $[0,\infty)$, 
namely, $x_{1}:=0$ and $x_{2}:=b\ln(|a|)$, and $\lim_{x\to\infty}f(x)=\infty$. 
This implies our claims (c)-(f) using the intermediate value theorem again.

$(g)$-$(j)$ Again, we compute the extrema of $f$. We have 
\[
 f'(x)=a(1-\frac{x}{b})e^{-x/b}+1\quad\text{and}\quad
 f''(x)=-\frac{a}{b}(2-\frac{x}{b})e^{-x/b},\;x\in(0,\infty).
\]
Defining $z:=1-\frac{x}{b}<0$ for $x>b$, we note that
\[
 0=f'(x)\quad\Leftrightarrow\quad -\frac{e}{a}=ze^z .
\]
Since $z<0$ for $x>b$ and $-\frac{e}{a}<0$, there are two solutions $z_0:=W_{0}(-e/a)$ 
and $z_{1}:=W_{-1}(-e/a)$ iff $-\frac{e}{a}\geq -e^{-1}$, which is equivalent to $a\geq e^2$. 
We remark that $z_0=z_1$ if $a=e^2$. By resubstitution we obtain that $0=f'(x)$ 
has the solutions $x_{0}:=b(1-W_{0}(-e/a))$ and $x_{1}:=b(1-W_{-1}(-e/a))$ iff $a\geq e^2$. 
From $-1< W_{0}(-e/a)<0$ and $-\infty<W_{-1}(-e/a)<-1$ for $a>e^2$ follows 
$b<x_{0}<2b$ and $x_{1}>2b$, which implies 
\[
f''(x_{0})=\underbrace{-\frac{a}{b}}_{<0}\underbrace{(2-\frac{x_{0}}{b})e^{-x_{0}/b}}_{>0}<0
\]
and 
\[
f''(x_{1})=\underbrace{-\frac{a}{b}}_{<0}\underbrace{(2-\frac{x_{1}}{b})e^{-x_{1}/b}}_{<0}>0.
\]
Hence $x_{0}$ is a local maximum and $x_{1}$ a local minimum if $a>e^2$. If $a=e^2$, then 
$x_{0}=x_{1}=2b$ and $f''(2b)=0$ and $f'''(2b)=a/(e\cdot b)^{2}\neq 0$, 
yielding that $x_{0}$ is a saddle point.
In addition, we observe that $f(0)=0$ and 
\[
\lim_{x\to\infty}f(x)=\lim_{x\to\infty}x+\lim_{x\to\infty}ae^{-x/b}x=\infty+0=\infty.
\]
This proves our claims (g)-(j) by the intermediate value theorem, 
in particular, we deduce that $f$ is strictly increasing on $[0,\infty)$ if $0\leq a\leq e^{2}$.

$(k)$ The first derivative for all $x\in(0,\infty)$ is given by
\[
 f'(x)=a(1-\frac{x}{b})e^{-x/b}+1.
\]
As $f(0)=0$ and $\lim_{x\to\infty}f(x)=\infty$, we only need to 
show that $f$ is strictly increasing on $[0,\infty)$ due to the intermediate value theorem. 
Since 
\[
0<e^{-x/b}<1\quad\text{and}\quad 
1-\frac{x}{b}<1
\]
for all $x>0$, we derive that 
\[
f'(x)=a(1-\frac{x}{b})e^{-x/b}+1>a+1\geq 0
\] 
for all $x>0$, confirming our claim.

$(l)$-$(o)$ Since $f(0)=f(b\ln(|a|))=0$ and $\lim_{x\to\infty}f(x)=\infty$, we only need to 
show that $f$ is strictly increasing on $(b\ln(|a|),\infty)$ and has a unique local maximum 
on $(0,b\ln(|a|))$. Then the intermediate value theorem proves our claim. The first derivative is 
\[
 f'(x)=\begin{cases}-a(1-\frac{x}{b})e^{-x/b}-1&,\;x\in(0,b\ln(|a|)),\\a(1-\frac{x}{b})e^{-x/b}+1&,\;x\in(b\ln(|a|),\infty).\end{cases}
\] 
We start with the proof that $f$ is strictly increasing on $(b\ln(|a|),\infty)$. 
We have for $x\in(b\ln(|a|),\infty)$ that
\[
0<e^{-x/b}< e^{-b\ln(|a|)/b}=|a|^{-1}\quad\text{and}\quad 
1-\frac{x}{b}< 1-\frac{b\ln(|a|)}{b}=1-\ln(|a|),
\]
implying
\begin{align*}
f'(x)&=a(1-\frac{x}{b})e^{-x/b}+1\geq a(1-\ln(|a|))|a|^{-1}+1=-1+\ln(|a|)+1\\
&=\ln(|a|)>0.
\end{align*}
Hence $f$ is strictly increasing on $(b\ln(|a|),\infty)$. Let us turn to the local maximum on $(0,b\ln(|a|))$. 
Setting $z:=1-\frac{x}{b}>0$ for $0<x<b$, we have the equivalences
\[
 0=f'(x)\;\Leftrightarrow\; 0=-aze^{-1}e^{z}-1\;\Leftrightarrow\; -\frac{e}{a}=ze^{z}.
\]
Since $z>0$ and $-\frac{e}{a}>0$, the unique solution of the last equation is given by $z=W_{0}(-e/a)$ 
This yields by resubstitution $x_{0}:=b(1-W_{0}(-e/a))$ and $x_{0}\in(0,b\ln(|a|))$ by Rolle's theorem. 
In addition, we observe that $x_{0}=b(1-W_{0}(-e/a))<b$, giving
\[
 f''(x_{0})=\underbrace{\frac{a}{b}}_{<0}\underbrace{(2-\frac{x_{0}}{b})e^{-x_{0}/b}}_{>0}<0. 
\]
Therefore $x_{0}$ is a local maximum on $(0,b\ln(|a|))$.
\begin{center}
\includegraphics[scale=0.25]{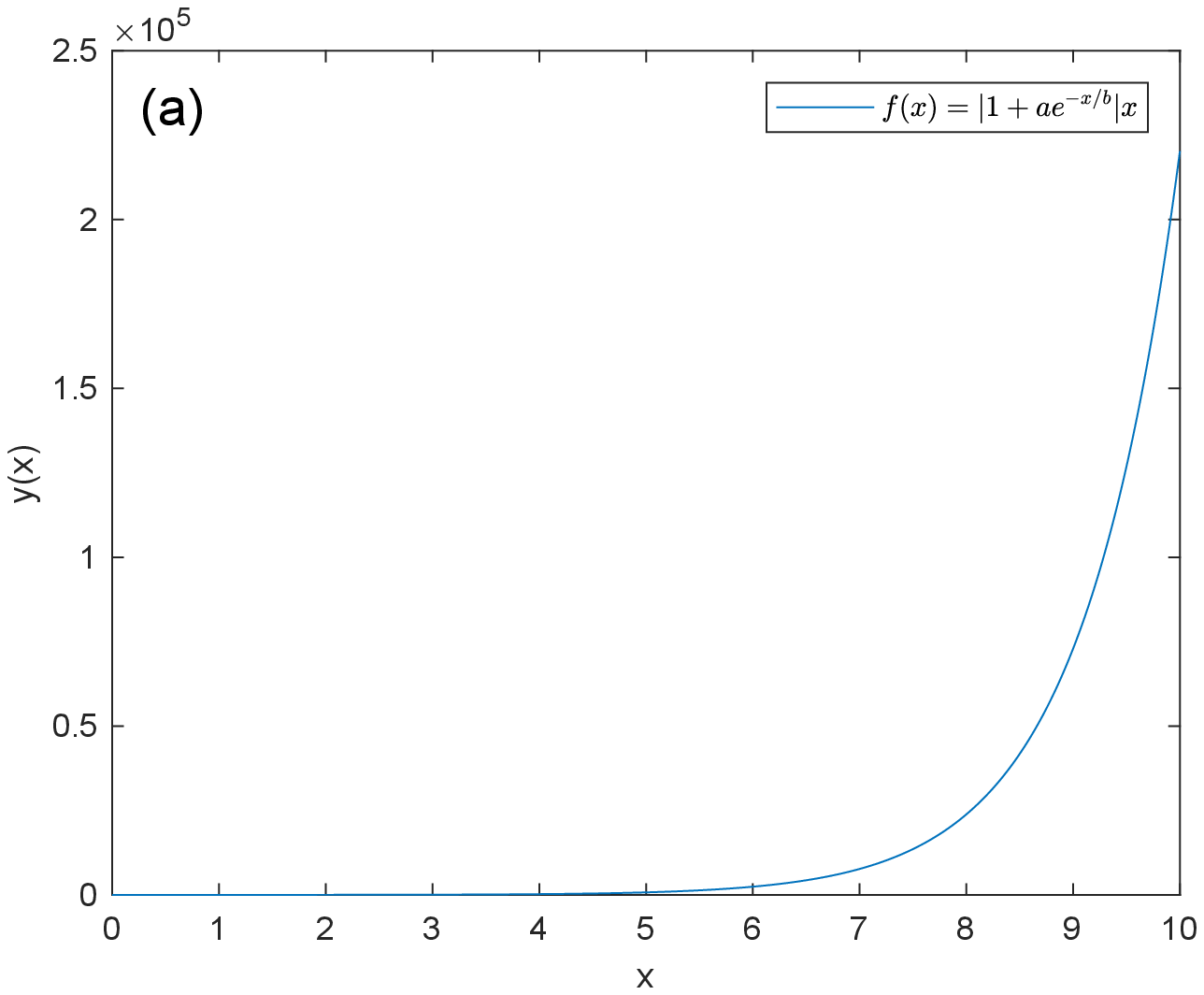}
\includegraphics[scale=0.25]{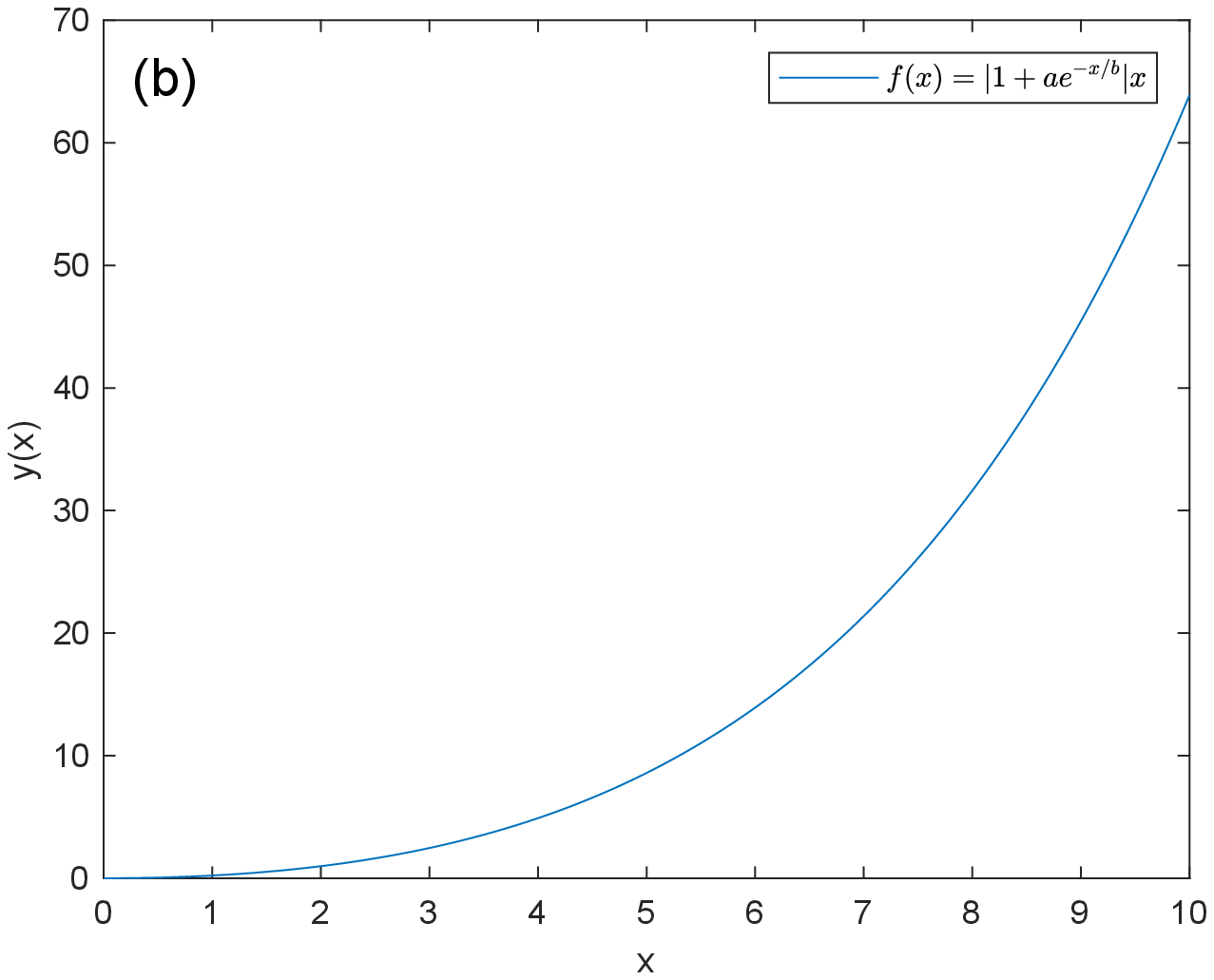}
\includegraphics[scale=0.25]{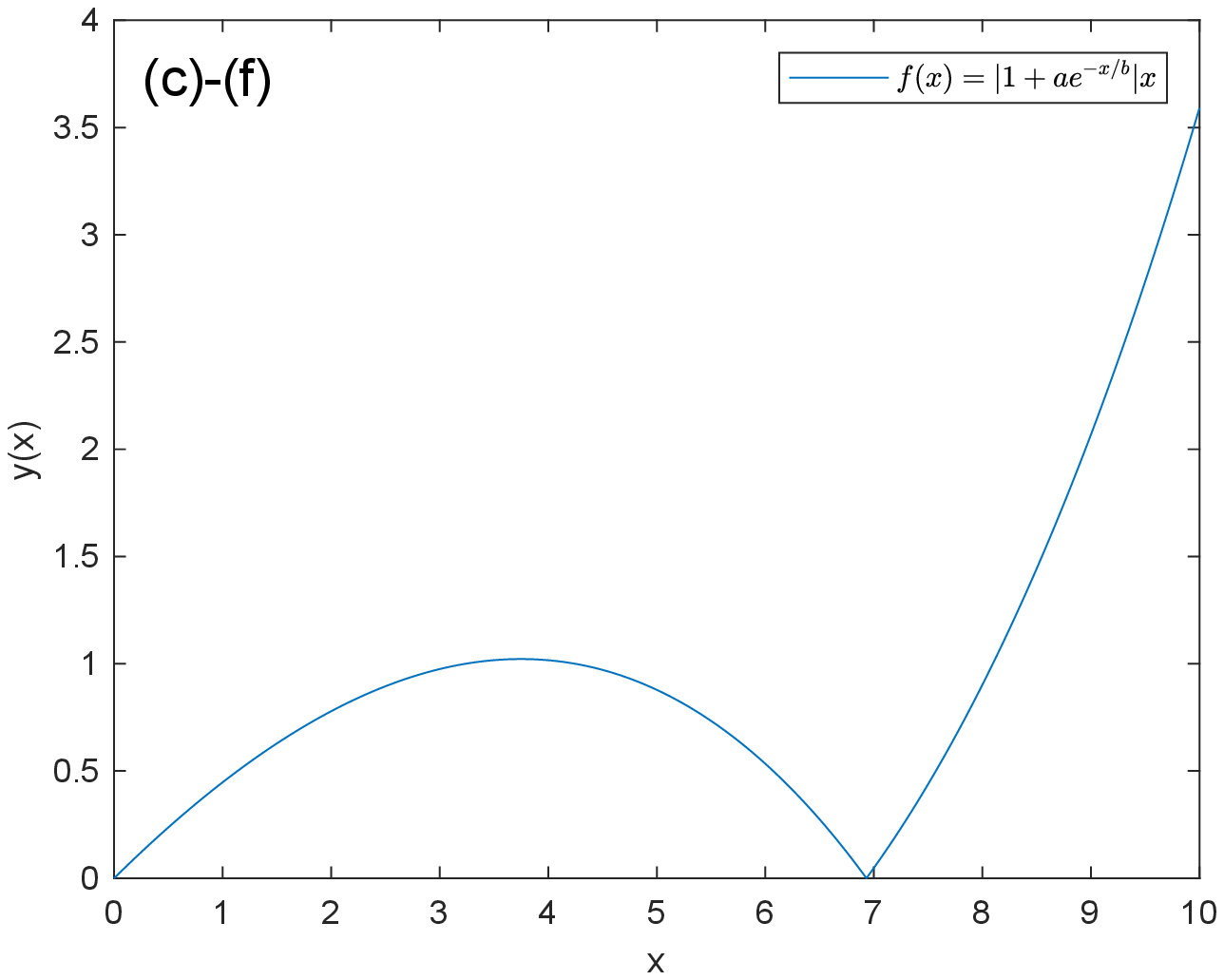}
\includegraphics[scale=0.25]{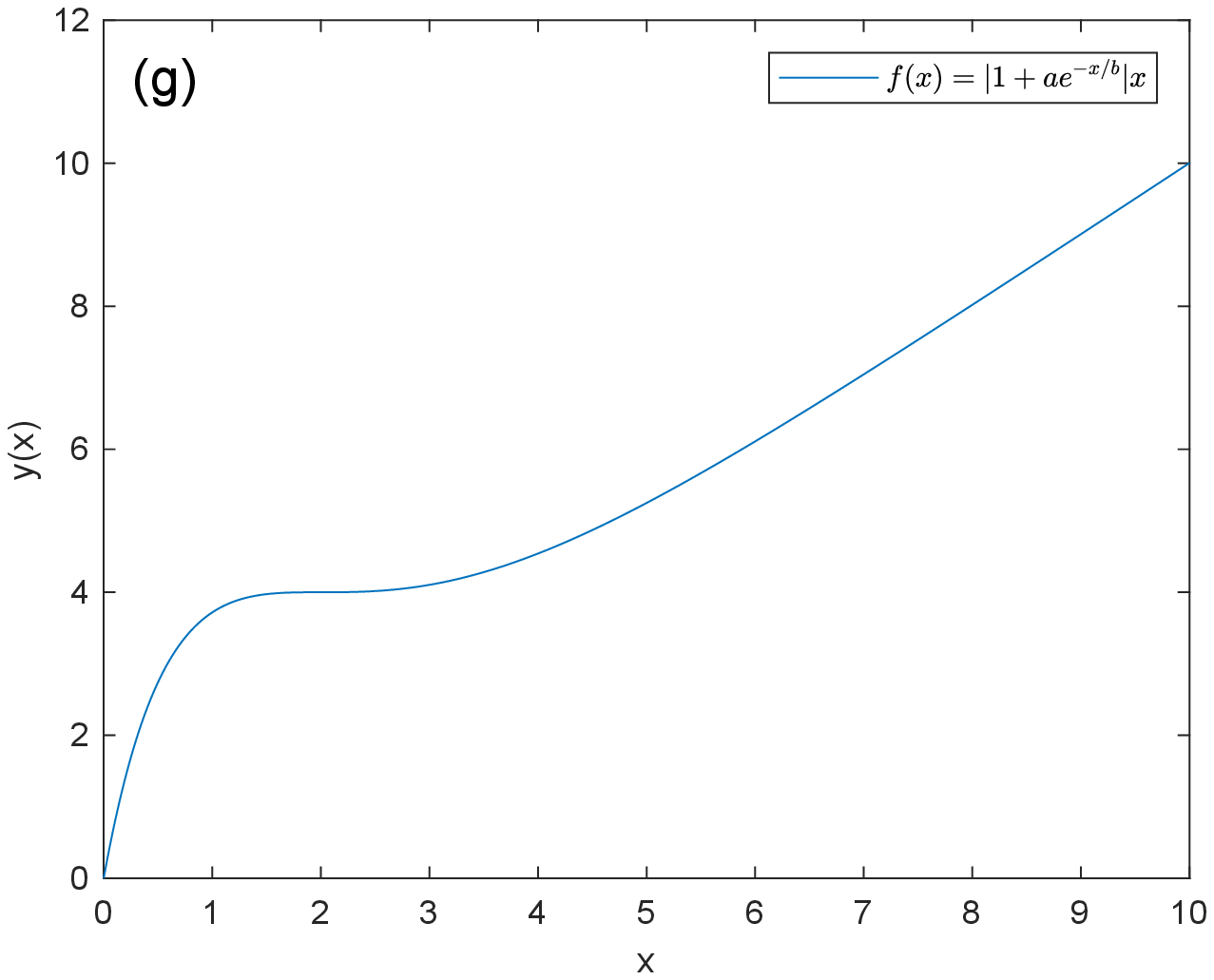}
\includegraphics[scale=0.25]{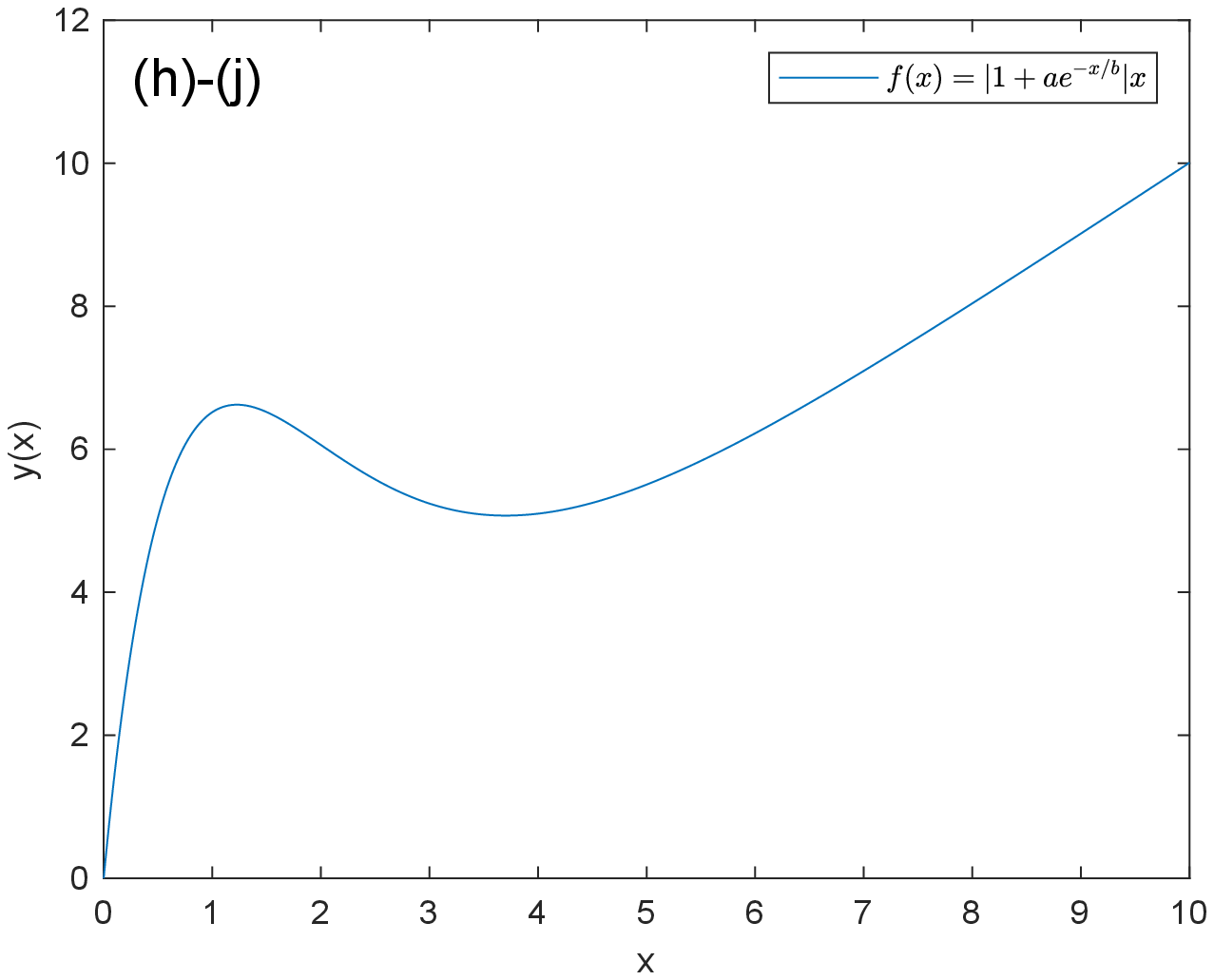}
\includegraphics[scale=0.25]{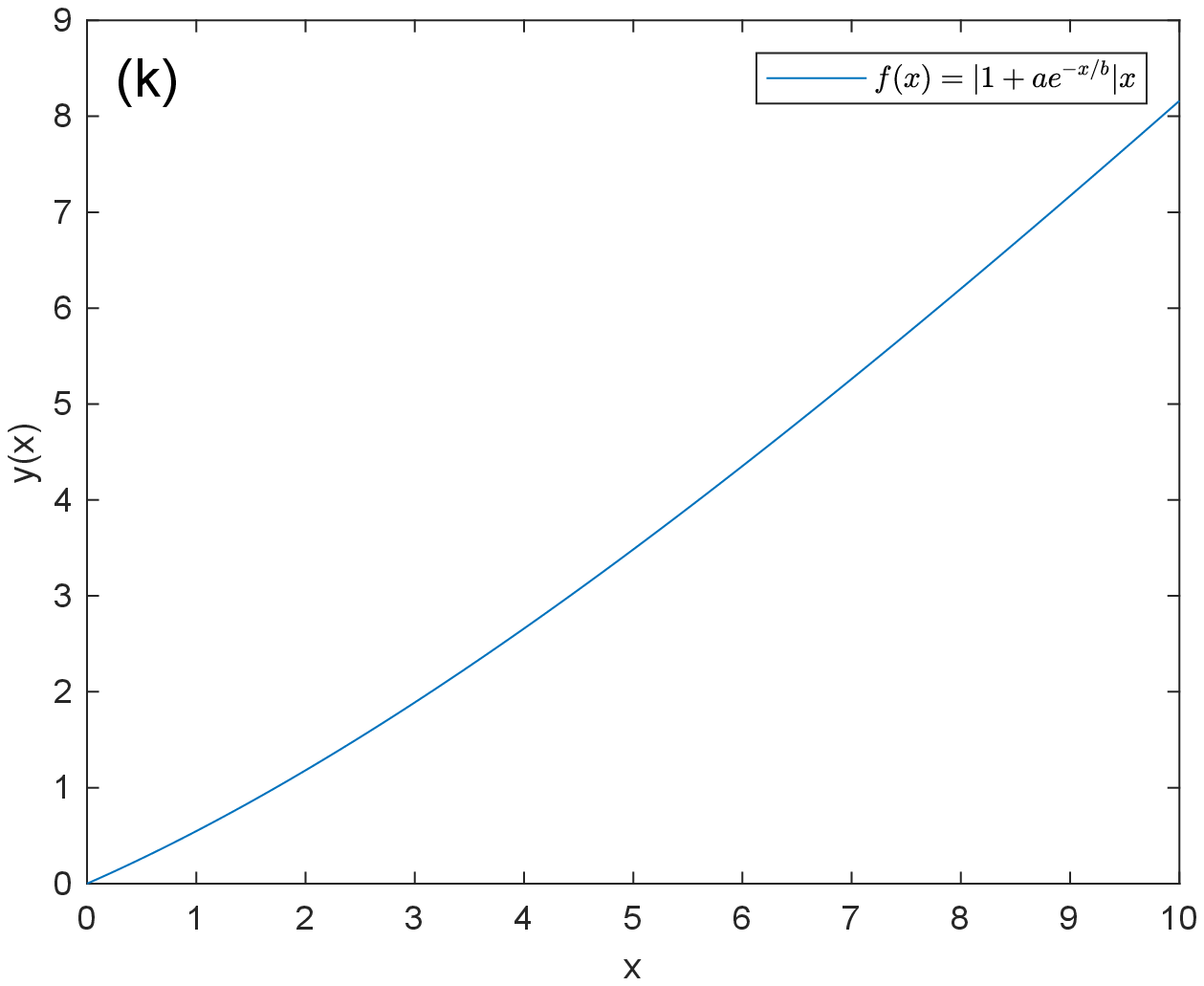}
\includegraphics[scale=0.25]{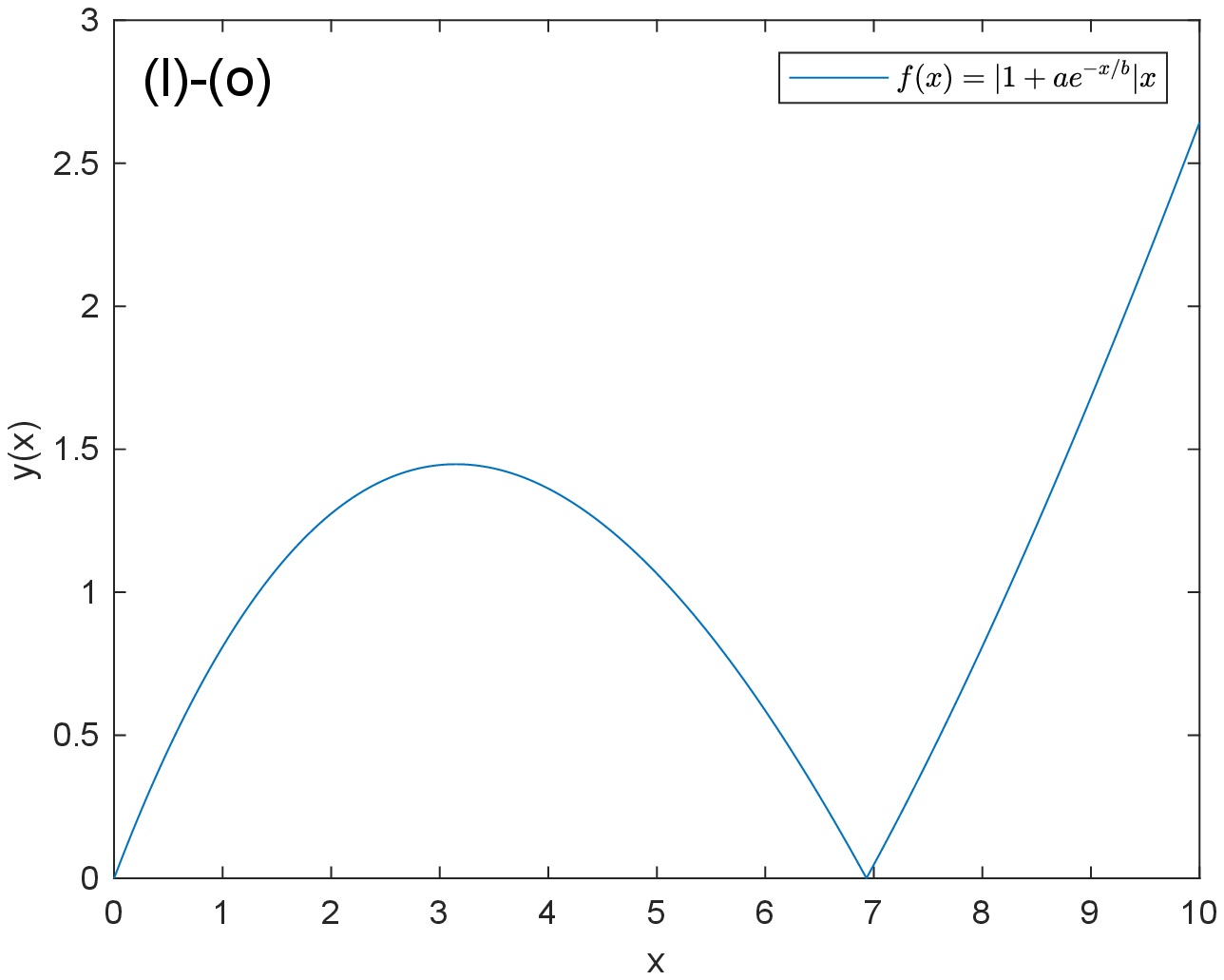}
\captionof{figure}{Graph of $f$ with (a) $a=1$ and $b=-1$, (b) $a=-1$ and $b=-5$, (c)-(f) $a=-1/2$ and $b=-10$, 
(g) $a=e^2$ and $b=1$, (h)-(j) $a=15$ and $b=1$, (k) $a=-1/2$ and $b=10$, (l)-(o) $a=-2$ and $b=10$.}
\end{center}
\end{proof}

Due to the preceding propositions we obtain a solution of our problem. 

\begin{thm}\label{thm:matrix_case}
Let $(E,\|\cdot\|)$ be a normed space, $E\neq\{0\}$, $a,b\in\R$, $b\neq 0$, and $Y\in E$. Then the
equation 
\[
(1+ae^{-\frac{\|X\|}{b}})X=Y
\]
has in $E$
\begin{enumerate} 
 \item[(a)] a unique solution if $a\geq 0$ and $b<0$,
 \item[(b)] a unique solution if $a\leq -1$ and $b<0$,
 \item[(c)] the set of solutions $\{0\}\cup\{X\in E\;|\;\|X\|=b\ln(|a|)\}$ if $-1<a<0$, $b<0$ and $Y=0$,
 \item[(d)] three solutions if $-1<a<0$, $b<0$ and $0<\|Y\|<(1+ae^{W_{0}(-e/a)-1})b(1-W_{0}(-e/a))$,
 \item[(e)] two solutions if $-1<a<0$, $b<0$ and $\|Y\|=(1+ae^{W_{0}(-e/a)-1})b(1-W_{0}(-e/a))$,
 \item[(f)] a unique solution if $-1<a<0$, $b<0$ and $\|Y\|>(1+ae^{W_{0}(-e/a)-1})b(1-W_{0}(-e/a))$,
 \item[(g)] a unique solution if $0\leq a\leq e^{2}$ and $b>0$,
 \item[(h)] a unique solution if $a>e^{2}$, $b>0$ and $\|Y\|<(1+ae^{W_{-1}(-e/a)-1})b(1-W_{-1}(-e/a))$ 
 or $\|Y\|>(1+ae^{W_{0}(-e/a)-1})b(1-W_{0}(-e/a))$,
 \item[(i)] two solutions if $a>e^{2}$, $b>0$ and $\|Y\|=(1+ae^{W_{-1}(-e/a)-1})b(1-W_{-1}(-e/a))$ 
 or $\|Y\|=(1+ae^{W_{0}(-e/a)-1})b(1-W_{0}(-e/a))$,
 \item[(j)] three solutions if $a>e^{2}$, $b>0$ and 
 $(1+ae^{W_{-1}(-e/a)-1})b(1-W_{-1}(-e/a))<\|Y\|<(1+ae^{W_{0}(-e/a)-1})b(1-W_{0}(-e/a))$,
 \item[(k)] a unique solution if $-1\leq a<0$ and $b>0$,
 \item[(l)] the set of solutions $\{0\}\cup\{X\in E\;|\;\|X\|=b\ln(|a|)\}$ if $a<-1$, $b>0$ and $Y=0$,
 \item[(m)] three solutions if $a<-1$, $b>0$ and $0<\|Y\|<-(1+ae^{W_{0}(-e/a)-1})b(1-W_{0}(-e/a))$,
 \item[(n)] two solutions if $a<-1$, $b>0$ and $\|Y\|=-(1+ae^{W_{0}(-e/a)-1})b(1-W_{0}(-e/a))$,
 \item[(o)] a unique solution if $a<-1$, $b>0$ and $\|Y\|>-(1+ae^{W_{0}(-e/a)-1})b(1-W_{0}(-e/a))$.
\end{enumerate}
\end{thm}
\begin{proof}
All cases except for (c) and (l) are a direct consequence of our \prettyref{prop:scalar_case} 
and \prettyref{prop:relation_scalar_case}. 
Now, we turn to the cases (c) and (l). First, we observe that $b\ln(|a|)>0$. Now, we only need to 
remark that $(1+ae^{-\frac{\|X\|}{b}})X=0$ for some $X\in E$ if and only if $X=0$ or $\|X\|=b\ln(|a|)$. 
\end{proof}

Part (a) implies that our motivating equation \eqref{eq:stress_discr_equiv} from the introduction is uniquely solvable.

\begin{rem}\label{rem:numerics}
\begin{enumerate}
\item[a)] \prettyref{prop:scalar_case} and \prettyref{thm:matrix_case} remain valid 
if $\|\cdot\|$ is replaced by an absolutely $\R$-homogeneous function $p\colon E\to [0,\infty)$, i.e.\ 
a function $p$ such that $p(\lambda X)=|\lambda|p(X)$ for all $\lambda\in\R$.
\item[b)] There are at least three solutions in (c) and (l) if $E\neq\{0\}$. 
Namely, choosing $\widetilde{X}\in E$ with $\|\widetilde{X}\|=1$, we always have the solutions $X=0$, $X=b\ln(|a|)\widetilde{X}$ 
or $X=-b\ln(|a|)\widetilde{X}$. If $\K=\C$, then there are infinitely many $\widetilde{X}\in E$, $\|\widetilde{X}\|=1$, and  
so there are infinitely many solutions in (c) and (l).
\item[c)] Let $E:=\R^{m\times n}$, $m,n\in\N$, and denote by 
$\|\cdot\|_{p}$ with $p\in[1,\infty]$ the $p$-norm and by $\|\cdot\|_{F}$ the Frobenius norm on $E$. 
If $\|\cdot\|\in\{\|\cdot\|_{1},\|\cdot\|_{2},\|\cdot\|_{\infty},\|\cdot\|_{F}\}$ and $m,n\in\N$, $m,n\geq 2$,
then there are infinitely many solutions in case (c) and (l).
It suffices to show that there are infinitely many matrices $X\in\R^{m\times n}$ 
with $\|X\|=b\ln(|a|)$. We choose $c\in\R$ with $0\leq c\leq b\ln(|a|)$.
If $\|\cdot\|=\|\cdot\|_{F}$, we define $X:=(x_{ij})\in\R^{m\times n}$ 
by $x_{11}:=\sqrt{b^{2}\ln^{2}(|a|)-c^2}$, $x_{12}:=c$ and $x_{ij}:=0$ else. 
Then we have $\|X\|_{F}=\sqrt{x_{11}^2+x_{12}^{2}}=b\ln(|a|)$, 
which implies 
\[
(1+ae^{-\frac{\|X\|_{F}}{b}})X=0\cdot X=0.
\]
Since the maximal singular value of $X$ is $\sqrt{x_{11}^2+x_{12}^{2}}=b\ln(|a|)$, we have $\|X\|_{2}=b\ln(|a|)=\|X\|_{F}$. 
Thus we have infinitely many solutions if $\|\cdot\|$ is the $2$-norm or the Frobenius norm. 
If $\|\cdot\|=\|\cdot\|_{1}$, we define $X:=(x_{ij})\in\R^{m\times n}$ 
by $x_{11}:=b\ln(|a|)-c$, $x_{21}:=c$ and $x_{ij}:=0$ else. 
If $\|\cdot\|=\|\cdot\|_{\infty}$, we define $X:=(x_{ij})\in\R^{m\times n}$ 
by $x_{11}:=b\ln(|a|)-c$, $x_{12}:=c$ and $x_{ij}:=0$ else. Then $\|X\|=b\ln(|a|)$ if $\|\cdot\|$ is the $1$-norm or $\infty$-norm, 
which again implies that there are infinitely many solutions.

If $n=m=1$, then there are only the  three solutions
$X=0$, $X=b\ln(|a|)$ or $X=-b\ln(|a|)$ from b).
\item[d)] Let $E\neq\{0\}$, $Y\in E$ and $y:=\|Y\|$. If $(x_{k})_{k\in\N}$ is a sequence in $[0,\infty)$ which converges 
to a solution $x\in[0,\infty)$ of $|1+ae^{-x/b}|x=y$ with $1+ae^{-x/b}\neq 0$, 
then the sequence given by $X_{k}:=(1+ae^{-x_{k}/b})^{-1}Y\in E$ is well-defined if $k$ is big enough and
converges to the solution $X:=(1+ae^{-x/b})^{-1}Y\in E$ of $(1+ae^{-\|X\|/b})X=Y$ w.r.t.\ $\|\cdot\|$ because 
\begin{align*}
\|X_{k}-X\|&=\|(1+ae^{-x_{k}/b})^{-1}Y-(1+ae^{-x/b})^{-1}Y\|\\
&=|(1+ae^{-x_{k}/b})^{-1}-(1+ae^{-x/b})^{-1}|\|Y\|
\end{align*}
and $\lim_{k\to\infty}(1+ae^{-x_{k}/b})^{-1}=(1+ae^{-x/b})^{-1}$. 
Thus, if we use a numerical method which produces a sequence $(x_{k})_{k\in\N}$ with $\lim_{k\to\infty}x_{k}=x$, 
we obtain by $(X_{k})_{k\in\N}$ a sequence of matrices which converges to the solution $X$. 
\item[e)] In the MATLAB \cite{MATLAB2017} \texttt{m-file} named \texttt{msolve} (see \prettyref{app:m_file}) 
our propositions and theorem are used to solve the norm equation 
\eqref{eq:scalar_eq} with $y:=\|Y\|_{F}$ for given inputs $a,b\in\R$, $b\neq 0$, and $Y\in\R^{m\times n}=:E$ w.r.t.\ 
to the Frobenius norm. 
The MATLAB function \texttt{fzero} is applied to compute the roots of $g_{y}(x):=|1+ae^{-x/b}|x-y$ in $[0,\infty)$. 
Then these roots are used to obtain the solutions of the matrix equation \eqref{eq:matrix_eq}
via formula \eqref{eq:solution_matrix_eq}.
\item[f)] Instead of \texttt{fzero} one might use the Newton-Raphson method to compute the roots of $g_{y}$, at least, 
in some of our cases. One of the difficulties of the Newton-Raphson method is the choice of suitable initial values 
$x_{0}\in[0,\infty)$ for the Newton iteration. For example in case (a) this can be solved since $g_{y}'(x)=f'(x)>0$ and 
$g_{y}''(x)=f''(x)>0$ for all $x\in[0,\infty)$, which implies that the Newton-Raphson method converges by 
\cite[Satz 30, p.\ 229]{Spaeth1994} for every initial value $x_{0}\in[0,\infty)$ with $x_{0}>x^{\ast}$ 
where $x^{\ast}$ is the root of $g_{y}$ in $[0,\infty)$. Using that 
\[
g_{y}(y)=(1+ae^{-y/b})y-y=ae^{-y/b}y>0
\]
and that $g_{y}$ is strictly increasing in case (a), we can choose $x_{0}:=y$ as a suitable initial value. 
\end{enumerate}
\end{rem}
\appendix
\section{\texttt{m-file} to compute the solutions of $(1+ae^{-\frac{\|X\|}{b}})X=Y$}\label{app:m_file}
As mentioned in \prettyref{rem:numerics} e) the solutions of $(1+ae^{-\frac{\|X\|}{b}})X=Y$ in 
$E=\R^{m\times n}$, $m,n\in\N$, $m,n\geq 2$, for the Frobenius norm 
$\|\cdot\|=\|\cdot\|_{F}$ are computed in MATLAB using the following \texttt{m-file} named \texttt{msolve}.
\begin{lstlisting}
function X=msolve(Y,a,b)
y=norm(Y,'fro'); % other norms like norm(Y,1), norm(Y,2) 
                 % or norm(Y,Inf) are possible as well                                                          
if (nargin ~= 3 || b==0 || min(size(Y))<2)
    error('Inputs have the wrong dimensions!');
end

% Simple cases where no calculation is needed

if (y==0) && ((-1<a && a<0 && b<0) || (a<-1 && b >0)) % cases (c) and (l)
    fprintf('Infinitely many solutions!');
elseif (y==0) % special case of the other cases
    X = Y;
elseif (y>0 && a==0) % special case of (a) and (g)
    X = Y;
end

% Simple cases with unique solutions 

if (y>0) && (0<a && b<0)                                    % case (a)
    x = fzero(@(x) (1+a*exp(-x/b))*x-y, y);
    X = (1/(1+a*exp(-x/b)))*Y; 
elseif (y>0) && (a<=-1 && b<0)                              % case (b)
    x = fzero(@(x) -(1+a*exp(-x/b))*x-y, y-b*log(2));
    X = (1/(1+a*exp(-x/b)))*Y; 
elseif (y>0) && (a>0 && a<=exp(2) && b>0)                   % case (g)
    x = fzero(@(x) (1+a*exp(-x/b))*x-y,y);
    X = (1/(1+a*exp(-x/b)))*Y;
elseif (y>0) && (-1<=a && a<0 && b>0)                       % case (k)
    x = fzero(@(x) (1+a*exp(-x/b))*x-y, 2*y+b*log(-2*a));
    X = (1/(1+a*exp(-x/b)))*Y;  
end

% Lambert W function for the more complicated cases

if (y>0) && ((-1<a && a<0 && b<0) || (a<-1 && b>0))% for cases (d)-(f) 
                                                   % and (m)-(o)
    z0  = lambertw(0,-exp(1)/a);          % z0 = W_0(-e/a)
    fz0 = (1+a*exp(z0-1))*b*(1-z0);       % fz0 = f(b(1-W_0(-e/a)))
elseif (y>0) && (a>exp(2) && b>0)         % for case (h)-(j) 
    z0  = lambertw(0,-exp(1)/a);          % z0 = W_0(-e/a)
    z1  = lambertw(-1,-exp(1)/a);         % z1 = W_{-1}(-e/a)
    fz0 = (1+a*exp(z0-1))*b*(1-z0);       % fz0 = f(b(1-W_0(-e/a)))
    fz1 = (1+a*exp(z1-1))*b*(1-z1);       % fz1 = f(b(1-W_{-1}(-e/a)))
end

% The more complicated cases 

if  (-1<a && a<0 && b<0) && (y>0 && y<fz0)                  % case (d)
    xl = fzero(@(x) (1+a*exp(-x/b))*x-y,[0,b*(1-z0)]);
    Xl = (1/(1+a*exp(-xl/b)))*Y; 
    xm = fzero(@(x) (1+a*exp(-x/b))*x-y,[b*(1-z0),b*log(-a)]);
    Xm = (1/(1+a*exp(-xm/b)))*Y; 
    xr = fzero(@(x) -(1+a*exp(-x/b))*x-y,y+b*log(-a/2));
    Xr = (1/(1+a*exp(-xr/b)))*Y; 
    X  = [Xl,Xm,Xr]; % three solutions stored in one matrix
elseif (-1<a && a<0 && b<0) && (y>0 && y==fz0)              % case (e)             
    Xl = (1/(1+a*exp(z0-1)))*Y; 
    xr = fzero(@(x) -(1+a*exp(-x/b))*x-y,y+b*log(-a/2));
    Xr = (1/(1+a*exp(-xr/b)))*Y;
    X  = [Xl,Xr];    % two solutions stored in one matrix
elseif (-1<a && a<0 && b<0) && (y>0 && y>fz0)               % case (f)
    x = fzero(@(x) -(1+a*exp(-x/b))*x-y,y+b*log(-a/2));
    X= (1/(1+a*exp(-x/b)))*Y;
elseif (a>exp(2) && b>0) && (y>0) && (y>fz0)                % case (h)
    x = fzero(@(x) (1+a*exp(-x/b))*x-y,[b*(1-z1),y]);
    X = (1/(1+a*exp(-x/b)))*Y; 
elseif (a>exp(2) && b>0) && (y>0) && (y<fz1)                % case (h)
    x = fzero(@(x) (1+a*exp(-x/b))*x-y,[0,b*(1-z0)]);
    X = (1/(1+a*exp(-x/b)))*Y;     
elseif (a>exp(2) && b>0) && (y>0) && (y==fz0)               % case (i)
    Xl = (1/(1+a*exp(z0-1)))*Y;
    xr = fzero(@(x) (1+a*exp(-x/b))*x-y,[b*(1-z1),y]);
    Xr = (1/(1+a*exp(-xr/b)))*Y;
    X  = [Xl,Xr];    % two solutions stored in one matrix
elseif (a>exp(2) && b>0) && (y>0) && (y==fz1)               % case (i)
    Xr = (1/(1+a*exp(z1-1)))*Y; 
    xl = fzero(@(x) (1+a*exp(-x/b))*x-y,[0,b*(1-z0)]);
    Xl = (1/(1+a*exp(-xl/b)))*Y; 
    X  = [Xl,Xr];    % two solutions stored in one matrix
elseif (a>exp(2) && b>0) && (y>0) && (y<fz0 && y>fz1)       % case (j)  
    xl = fzero(@(x) (1+a*exp(-x/b))*x-y,[0,b*(1-z0)]);
    Xl = (1/(1+a*exp(-xl/b)))*Y; 
    xm = fzero(@(x) (1+a*exp(-x/b))*x-y,[b*(1-z0),b*(1-z1)]);
    Xm = (1/(1+a*exp(-xm/b)))*Y;  
    xr = fzero(@(x) (1+a*exp(-x/b))*x-y,[b*(1-z1),y]);
    Xr = (1/(1+a*exp(-xr/b)))*Y; 
    X  = [Xl,Xm,Xr]; % three solutions stored in one matrix
elseif (a<-1 && b>0) && (y>0 && y<-fz0)                     % case (m)
    xl = fzero(@(x) -(1+a*exp(-x/b))*x-y,[0,b*(1-z0)]);
    Xl = (1/(1+a*exp(-xl/b)))*Y; 
    xm = fzero(@(x) -(1+a*exp(-x/b))*x-y,[b*(1-z0),b*log(-a)]);
    Xm = (1/(1+a*exp(-xm/b)))*Y; 
    xr = fzero(@(x) (1+a*exp(-x/b))*x-y, 2*y+b*log(-2*a));
    Xr = (1/(1+a*exp(-xr/b)))*Y; 
    X  = [Xl,Xm,Xr]; % three solutions stored in one matrix
elseif  (a<-1 && b>0) && (y>0 && y==-fz0)                   % case (n)             
    Xl = (1/(1+a*exp(z0-1)))*Y; 
    xr = fzero(@(x) (1+a*exp(-x/b))*x-y, 2*y+b*log(-2*a));
    Xr = (1/(1+a*exp(-xr/b)))*Y;
    X  = [Xl,Xr];    % two solutions stored in one matrix
elseif (a<-1 && b>0) && (y>0 && y>-fz0)                     % case (o)
    x = fzero(@(x) (1+a*exp(-x/b))*x-y, 2*y+b*log(-2*a));
    X = (1/(1+a*exp(-x/b)))*Y;
end

end
\end{lstlisting}
\bibliographystyle{plain}
\bibliography{Lit}
\end{document}